\documentclass[a4paper,oneside,11pt]{amsart}

\reversemarginpar

\usepackage[colorlinks,hyperindex,linkcolor=blue,urlcolor=black,pdftitle={On expansive operators that are quasisimilar
 to the unilateral shift of finite multiplicity},pdfauthor={Maria F.Gamal'}]
{hyperref}

\usepackage{amsmath}
\usepackage{amsfonts}
\usepackage{amssymb}
\usepackage{amsthm}

\usepackage[english]{babel}
\usepackage{enumerate}

\usepackage{graphicx}
\usepackage{color}

\usepackage[abbrev]{amsrefs}

\numberwithin{equation}{section}

\theoremstyle{plain}
\newtheorem{theorem}{Theorem}[section]
\newtheorem{lemma}[theorem]{Lemma}

\newtheorem{corollary}[theorem]{Corollary}

\newtheorem{theoremcite}{Theorem}

\theoremstyle{definition}

\begin{document}

\title[On existence of hyperinvariant subspaces]
{On existence of hyperinvariant subspaces for quasinilpotent operators with a nonsymmetry in the growth of the resolvent}

\author{Maria F. Gamal'}
\address{
 St. Petersburg Branch\\ V. A. Steklov Institute 
of Mathematics\\
Russian Academy of Sciences\\ Fontanka 27, St. Petersburg\\ 
191023, Russia  
}
\email{gamal@pdmi.ras.ru}



\begin{abstract} Let $T$ be a quasinilpotent operator on a Banach space. Under assumptions of 
 a certain nonsymmetry in the growth of  the resolvent of $T$, 
it is proved that every operator in the commutant of $T$ is not unicellular. 
In particular, $T$ has nontrivial hyperinvariant subspaces.
 The proof is based on a modification of the reasoning of \cite{sol}.  


 \end{abstract}

\maketitle

\section{Introduction}
 Let $\mathcal H$ be a (complex, separable) Banach space, 
and let  $\mathcal L(\mathcal H)$ be the algebra of all (bounded, linear)  operators acting on  $\mathcal H$. A (closed) subspace  $\mathcal M$ of  $\mathcal H$ is called \emph{invariant} 
for an operator $T\in\mathcal L(\mathcal H)$, if $T\mathcal M\subset\mathcal M$. 
 Denote by $\operatorname{Lat}T$ the collection of all invariant subspaces of $T$. It is well known and easy to see that $\operatorname{Lat}T$ is a complete lattice with 
the inclusion as partial ordering. An operator $T$ is called \emph{unicellular}, if $\operatorname{Lat}T$
is totally ordered by inclusion. 
 Set
$\{T\}'=\{A\in\mathcal L(\mathcal H)\ :\ AT=TA\}$ and $\{T\}''=\{A\in\mathcal L(\mathcal H)\ :\ AB=BA$ for every $B\in\{T\}'\}$.  It is well known and easy to see that $\{T\}'$ and $\{T\}''$ are algebras closed in weak operator topology. 
 A  subspace  $\mathcal M$ of  $\mathcal H$ is called \emph{hyperinvariant} 
for $T$,  if  $A\mathcal M\subset\mathcal M$ for every $A\in\{T\}'$. Denote by $\operatorname{Hlat}T$ the collection of all hyperinvariant subspaces of $T$. 
Then $\operatorname{Hlat}T$ is a complete lattice with 
the inclusion as partial ordering, $\operatorname{Hlat}T\subset\operatorname{Lat}T$, and $\ker A$, 
$\operatorname{clos}A\mathcal H\in\operatorname{Hlat}T$ for every $A\in\{T\}''$.

Hyperinvariant subspace problem is a question where every (bounded, linear)  operator acting on  a Hilbert space 
$\mathcal H$ and not equal to a scalar multiple of the identity operator $I$ has a nontrivial (that is, not equal to $\{0\}$ or to $\mathcal H$) hyperinvariant subspace.
On some Banach spaces there exist operators without nontrivial invariant subspaces,  
 see \cite{enflo}, \cite{read1}, \cite{read2}.

 It is well known that if $\sigma(T)$  (the spectrum of an operator $T$) is not connected, 
then nontrivial  hyperinvariant subspaces of  $T$ can be found by using 
the Riesz--Dunford functional calculus, see, for example, {\cite[Theorem 2.10]{rara}}. 
A similar method can be applied, if an operator $T$ has sufficiently rich spectrum and 
appropriate estimate of the norm of resolvent, see, for example,
 {\cite[Secs. 6.2 and 6.5]{rara}}, \cite{apostol}, \cite{berc}, {\cite[Sec. 4.1]{chalpart}}. 
In this paper, the case when  $\sigma(T)$ is a single point is considered. 
Under assumptions of  a certain nonsymmetry in the growth of  the resolvent of $T$, 
it is proved that   $\operatorname{Hlat}T$
is not totally ordered by inclusion. The proof is based on a modification of the reasoning of \cite{sol}. 
More precisely, a some case when the assumption of  \cite{sol} is not fulfilled for $T$ is considered (Theorem \ref{thmbeta}), and it is shown that 
$T$ satisfies to the assumptions of Lemma \ref{leminv} below.

An operator $T\in\mathcal L(\mathcal H)$ is called \emph{quasinilpotent}, if  $\sigma(T)=\{0\}$.
It is well known and  follows 
 directly from  the definition that 
 $\operatorname{Hlat}T=\operatorname{Hlat}(T-\lambda I)$ for every $T\in\mathcal L(\mathcal H)$ and every $\lambda\in\mathbb C$. 
Therefore, if $\sigma(T)=\{\lambda\}$, then one can consider a quasinilpotent operator $T-\lambda I$ instead of $T$. 
It is known (\cite{read3}) that on some Banach spaces there exist quasinilpotent operators without nontrivial invariant subspaces.  Since the formulation and the  proof of Theorem \ref{thmbeta} in the case of 
a Banach space $\mathcal H$ is the same as in the case of a Hilbert space,  Theorem \ref{thmbeta} is formulated and proved 
for a Banach space  $\mathcal H$.

 
 The following lemma is well known. We give a  proof to emphasize some details.

\begin{lemma}\label{leminv} Let $T\in\mathcal L(\mathcal H)$. Suppose that there exist  
 $A_1$, $A_2\in\{T\}''$ such that $A_1A_2=A_2A_1=\mathbb O$ and $A_k^2\neq\mathbb O$, $k=1,2$. 
Then $\operatorname{Hlat}T$ is not totally ordered by inclusion. In particular, every operator from $\{T\}'$ is not unicellular. 
\end{lemma}

\begin{proof} Set $\mathcal N_k=\ker A_k$ and $\mathcal M_k=\operatorname{clos}A_k\mathcal H$, $k=1,2$. 
Then $\mathcal M_k$, $\mathcal N_k\in\operatorname{Hlat}T$, because  $A_k\in\{T\}''$. Furthermore, 
 $\mathcal M_k\not\subset\mathcal N_k$, because $A_k^2\neq\mathbb O$, $k=1,2$.  Consequently,
 $\mathcal N_k\neq\mathcal H$ and $\mathcal M_k\neq\{0\}$, $k=1,2$.  
The equalities   $A_1A_2=\mathbb O$ and $A_2A_1=\mathbb O$ imply that  $\mathcal M_2\subset\mathcal N_1$
and $\mathcal M_1\subset\mathcal N_2$, respectively. Consequently,  $\mathcal N_k\neq\{0\}$ and  
$\mathcal M_k\neq\mathcal H$, $k=1,2$. 

To show that $\operatorname{Hlat}T$ is not totally ordered by inclusion we prove that $\mathcal M_k\not\subset\mathcal M_l$ and $\mathcal N_k\not\subset\mathcal N_l$, if $k\neq l$, $k,l=1,2$. 
Indeed, assume that $\mathcal M_1\subset\mathcal M_2$. Then $\mathcal M_1\subset \mathcal N_1$, a contradiction. 
Thus,  $\mathcal M_1\not\subset\mathcal M_2$. Similarly,  $\mathcal M_2\not\subset\mathcal M_1$. 
Assume that $\mathcal N_1\subset\mathcal N_2$. Then $\mathcal M_2\subset \mathcal N_2$, a contradiction.  
Thus,  $\mathcal N_1\not\subset\mathcal N_2$. Similarly,  $\mathcal N_2\not\subset\mathcal N_1$. 
\end{proof} 

The paper is organized as follows. In Sec. 2, auxiliary results are collected; most of them are modifications of results for operators with reach spectrum. In Sec. 3 Theorem \ref{thmbeta} is proved. In the end of Sec. 3, the comparison with a   result from \cite{sol} is given.

\section{Preliminaries}

The following definition and facts from {\cite[Ch. 10]{duren}} and {\cite[Ch. 7]{pom92}} will be used. 
 Let $\Gamma$ be a rectifiable Jordan curve, denote by $\Omega$ the bounded component of $\mathbb C\setminus \Gamma$; then $\Omega$ is a Jordan domain with boundary $\Gamma$. Denote  by $s$ the arc length measure on $\Gamma$. By $H^\infty(\Omega)$ the algebra of all bounded  analytic functions in $\Omega$ 
with the norm $\|f\|_{H^\infty(\Omega)}=\sup_{z\in\Omega}|f(z)|$
is denoted.  
By $E^1(\Omega)$  the Smirnov class  in $\Omega$ is denoted, see {\cite[Sec. 10.1]{duren}} for exact definition.  
Here we recall that if $f\in E^1(\Omega)$, then $f$ is a function  analytic in $\Omega$ and  $f$ has nontangential boundary values $f(z)$ for almost all 
$z\in\Gamma$ with respect to $s$.
 The function $z\mapsto f(z)$ for $z\in\Gamma$ will be denoted 
by the same letter $f$. If $f\in E^1(\Omega)$, then 
 $f\in L^1(\Gamma,s)$.
The inclusion $H^\infty(\Omega)\subset E^1(\Omega)$ is fulfilled. Furthermore, if $f\in E^1(\Omega)$, then
\begin{align*} f(w)=\frac{1}{2\pi\mathrm{i}}\int_\Gamma\frac{f(z)}{z-w}\mathrm{d}z \text{ for }w\in\Omega, 
\ \  
 \int_\Gamma\frac{f(z)}{z-w}\mathrm{d}z=0 \text{ for  }w\in\mathbb C\setminus\operatorname{clos}\Omega, 
\\ \text{ and } 
\int_\Gamma f(z)z^n\mathrm{d}z=0 \text{ for all } n\in\mathbb N\cup\{0\}.
\end{align*}

A Jordan  domain $\Omega$ with a rectifiable boundary $\Gamma$ is called a  \emph{Smirnov domain}, if the derivative of conformal mapping of the unit disc 
$\mathbb D$ onto $\Omega$ is an outer function. 
A  rectifiable Jordan curve $\Gamma$ is called a  \emph{Lavrentiev curve}  or \emph{chord-arc-curve}, if  there exists a constant $C$ such that 
$s(\gamma(z,w))\leq C|z-w|$ for every $z$, $w\in\Gamma$, where $\gamma(z,w)$ is the shorter arc of $\Gamma$ betweeen $z$ and $w$. 
If $\Gamma$ is a Lavrentiev curve, then $\Omega$ is a Smirnov domain by {\cite[Theorem 7.11]{pom92}}.

Let $\Omega$ be a Smirnov domain, let $f$ be a function analytic in $\Omega$, and let $f$ have nontangential boundary values $f(z)$ for almost all 
$z\in\Gamma$ with respect to $s$. Denote the function $z\mapsto f(z)$ for $z\in\Gamma$ by the same letter $f$. 
 If $f\in L^\infty(\Gamma,s)$ and $\int_\Gamma f(z)z^n\mathrm{d}z=0$ for all  $n\in\mathbb N\cup\{0\}$,  then $f\in H^\infty(\Omega)$ and 
\begin{equation*} \|f\|_{H^\infty(\Omega)}=\|f\|_{ L^\infty(\Gamma,s)}=
\mathop{\mathrm{ess\, sup}}_{z\in\Gamma}|f(z)| 
\end{equation*}
(where $\mathop{\mathrm{ess\, sup}}$ is taken with respect to $s$).
For  proof, see {\cite[Ch. 10.1--10.3]{duren}}.

For $\alpha\in(0,\pi/2)$ and two sequences $\{a_n\}_{n=1}^\infty$ and  $\{\delta_n\}_{n=1}^\infty$ such that
\begin{align*} 0<\ldots<a_{n+1}<a_n<\ldots<a_2<a_1=1, \ \ a_n\to 0, \\  
0< \delta_n<\min(\delta_{n-1}, a_{n+1}\sin\alpha) \text{  for } n\geq 2,\  \delta_1<a_2\sin\alpha, \end{align*} 
set
\begin{equation}\label{defgg}\mathcal G=\mathcal G(\alpha, \{a_n\}_{n=1}^\infty,\{\delta_n\}_{n=1}^\infty)=
\cup_{n=1}^\infty\{x+\mathrm{i}y\ : a_{n+1}\leq x\leq a_n, |y|\leq\delta_n\}.
\end{equation}
It follows from the definition of $\mathcal G$ that $\mathcal G\cup\{0\}$ is closed, and  
$\mathcal G\subset\{r\mathrm{e}^{\mathrm{i}t}\ :\ r>0, |t|<\alpha\}$.

\begin{lemma}\label{lemgg} Let $\alpha\in(0,\pi/2)$, and let 
 $\varphi\colon(0,1]\to(0,\infty)$ be a nonconstant continuous nonincreasing function.
Let $T\in\mathcal L(\mathcal H)$. Suppose that $(0,1]\subset\mathbb C\setminus\sigma(T)$ and there exists a constant  $C_0>0$  such that 
\begin{equation}\label{estgg} \|(zI-T)^{-1} \|\leq C_0\varphi(|z|) \text{ for every }z\in(0,1].
\end{equation}
Then  for every $C_1>C_0$ there exists 
a set $\mathcal G\subset \mathbb C\setminus\sigma(T)$ such that $\mathcal G$ has a form 
 \eqref{defgg} with $\alpha$ 
(and with appropriate sequences $\{a_n\}_{n=1}^\infty$ and  $\{\delta_n\}_{n=1}^\infty$) and 
\begin{equation}\label{estgg1} \|(zI-T)^{-1} \|\leq C_1\varphi(|z|)\text{ for every }z\in\mathcal G.
\end{equation}
\end{lemma}

\begin{proof} Let $z\in(0,1]$. Let $\varepsilon(z)>0$ be such that $\varepsilon(z)+C_0\varphi(|z|)<C_1\varphi(|z|)$. 
 There exists $\delta_0(z)>0$ such that 
$\{w\ :\ |w-z|<\delta_0(z)\}\subset \mathbb C\setminus\sigma(T)$ 
and \begin{align*}\|(wI-T)^{-1}\|\leq \varepsilon(z) +  \|(zI-T)^{-1} \|\leq \varepsilon(z) + C_0\varphi(|z|) \\
\text{ for every }w \text{ such that }|w-z|<\delta_0(z).
\end{align*}
Since $\varphi$ is continuous, there exists $\delta(z)$ such that $0<\delta(z)<\delta_0(z)$ and 
$\varepsilon(z) + C_0\varphi(|z|)\leq C_1\varphi(|z|+\delta(z))$.
Let $w$ be such that  $|z-w|\leq\delta(z)$. Then $|w|\leq |z|+\delta(z)$. Since $\varphi$ is nonincreasing, 
\begin{equation}\label{deltazw}
\begin{aligned}
\|(wI-T)^{-1}\|\leq C_1\varphi(|z|+\delta(z))\leq C_1\varphi(|w|)\\ \text{ for every }w \text{ such that }|w-z|<\delta(z).
\end{aligned}\end{equation}

Take a sequence $\{a_n\}_{n=1}^\infty$  such that
$0<\ldots<a_{n+1}<a_n<\ldots<a_2<a_1=1$, and $a_n\to 0$. Let $n\geq 1$. 
Then \begin{equation*}
[a_{n+1}, a_n]\subset\cup_{ z\in[a_{n+1}, a_n]}\{w\ :\ |z-w|<\delta(z)\},
\end{equation*}
where $\delta(z)$ is defined as above for every $z\in(0,1]$. Since $[a_{n+1}, a_n]$ is a compact set, 
there exists a finite subset $\{z_{nk}, k=1,\ldots, N_n\}\subset [a_{n+1}, a_n]$ such that 
\begin{equation*}
[a_{n+1}, a_n]\subset\cup_{k=1}^{N_n}\{w\ :\ |z_{nk}-w|<\delta(z_{nk})\}.
\end{equation*}
Set \begin{equation*} \delta_{1n}=\operatorname{dist}\Bigl([a_{n+1}, a_n], \ 
\partial\bigl(\cup_{k=1}^{N_n}\{w\ : \ |z_{nk}-w|<\delta(z_{nk})\}\bigr)\Bigr).
\end{equation*}
Then $\delta_{1n}>0$. 
Construct a sequence  $\{\delta_n\}_{n=1}^\infty$ satisfying the properties before \eqref{defgg} by induction. 
Take $\delta_1$ such that  $0\!<\!\delta_1\!<\!\min(\delta_{11}, a_2\sin\alpha)$. Assume that $\delta_{n-1}$ is constructed. 
Take $\delta_n$ such that  $0\!<\!\delta_n\!<\!\min(\delta_{1n},\delta_{n-1}, a_{n+1}\sin\alpha)$. 
Then \eqref{defgg} is fulfilled for $\mathcal G$ constructed by the sequences  $\{a_n\}_{n=1}^\infty$ and  $\{\delta_n\}_{n=1}^\infty$ 
due to \eqref{deltazw}.
\end{proof}

\begin{lemma}\label{lem0} Let $\Gamma$ be a rectifiable Jordan curve. Denote by $\Omega$ the bounded components of $\mathbb C\setminus \Gamma$ 
and by $s$ the arc length measure on $\Gamma$. 
Let $f\in E^1(\Omega)$, and let $T\in\mathcal L(\mathcal H)$. 
Suppose that 
\begin{equation}\label{est0} s(\sigma(T)\cap\Gamma)=0, \text{ and } 
\mathop{\mathrm{ess\, sup}}_{z\in\Gamma}|f(z)|\|(zI-T)^{-1} \|<\infty 
\end{equation}
(where $\mathop\mathrm{{ess\, sup}}$ is taken with respect to $s$).
Set
\begin{equation*} A=\frac{1}{2\pi\mathrm{i}}\int_{\Gamma}f(z)(zI-T)^{-1}\mathrm{d}z.
\end{equation*}
Then $A\in\{T\}''$ and 
\begin{equation}\label{ttnaa} T^nA=\frac{1}{2\pi\mathrm{i}}\int_{\Gamma}z^nf(z)(zI-T)^{-1}\mathrm{d}z \ \text{ for every } n\in\mathbb N.
\end{equation}
\end{lemma}

\begin{proof} The inclusion $A\in\mathcal L(\mathcal H)$ follows from \eqref{est0}, and $A\in\{T\}''$, because 
$f(z)(zI-T)^{-1}\in\{T\}''$ for $s$-a.e. $z\in\Gamma$. Prove \eqref{ttnaa} by induction. 
For $n=1$ we have 
\begin{align*} TA&=\frac{1}{2\pi\mathrm{i}}\int_{\Gamma}f(z)T(zI-T)^{-1}\mathrm{d}z=
\frac{1}{2\pi\mathrm{i}}\int_{\Gamma}f(z)(T-zI+zI)(zI-T)^{-1}\mathrm{d}z\\&=
\Bigl(-\frac{1}{2\pi\mathrm{i}}\int_{\Gamma}f(z)\mathrm{d}z\Bigr)I+
\frac{1}{2\pi\mathrm{i}}\int_{\Gamma}f(z)z(zI-T)^{-1}\mathrm{d}z\\&=
\frac{1}{2\pi\mathrm{i}}\int_{\Gamma}f(z)z(zI-T)^{-1}\mathrm{d}z,
\end{align*}
because $\int_{\Gamma}f(z)\mathrm{d}z=0$. 
Suppose that \eqref{ttnaa} is proved for $n\in\mathbb N$. 
We have 
\begin{align*} T^{n+1}A&=\frac{1}{2\pi\mathrm{i}}\int_{\Gamma}z^nf(z)T(zI-T)^{-1}\mathrm{d}z\\&=
\frac{1}{2\pi\mathrm{i}}\int_{\Gamma}z^nf(z)(T-zI+zI)(zI-T)^{-1}\mathrm{d}z\\&=
\Bigl(-\frac{1}{2\pi\mathrm{i}}\int_{\Gamma}z^nf(z)\mathrm{d}z\Bigr)I+
\frac{1}{2\pi\mathrm{i}}\int_{\Gamma}f(z)z^{n+1}(zI-T)^{-1}\mathrm{d}z\\&=
\frac{1}{2\pi\mathrm{i}}\int_{\Gamma}f(z)z^{n+1}(zI-T)^{-1}\mathrm{d}z,
\end{align*}
because $\int_{\Gamma}z^nf(z)\mathrm{d}z=0$.
\end{proof} 

 \begin{lemma}\label{lemkeraa} Let $\Gamma$ be a rectifiable Jordan curve. 
 Denote by $\Omega$ the bounded components of $\mathbb C\setminus \Gamma$ 
and by $s$ the arc length measure on $\Gamma$. Suppose that $\Omega$ is a Smirnov domain. 
Let $f\in E^1(\Omega)$, and let $T\in\mathcal L(\mathcal H)$. 
Suppose that $\sigma(T)\cap\Omega=\emptyset$, $s(\sigma(T)\cap\Gamma))=0$,  and 
\begin{equation}\label{estkeraa}
\mathop\mathrm{{ess\, sup}}_{z\in\Gamma}|f(z)|\|(zI-T)^{-1} \|<\infty
\end{equation}
(where $\mathop\mathrm{{ess\, sup}}$ is taken with respect to $s$).
Set
\begin{equation*} A=\frac{1}{2\pi\mathrm{i}}\int_{\Gamma}f(z)(zI-T)^{-1}\mathrm{d}z.
\end{equation*}
Then  
\begin{equation*} \ker A=\{x\in\mathcal H\ : \ \sup_{z\in\Omega}|f(z)|\|(zI-T)^{-1}x\|<\infty\}.
\end{equation*}
Furthermore, if $x\in\ker A$, then 
\begin{equation*}  \sup_{z\in\Omega}|f(z)|\|(zI-T)^{-1}x\|\leq
 \mathop\mathrm{{ess\, sup}}_{z\in\Gamma}|f(z)|\|(zI-T)^{-1} x\|
\end{equation*} 
(where $\mathop\mathrm{{ess\, sup}}$ is taken with respect to $s$).
\end{lemma}

\begin{proof} For  every $x\in\mathcal H$ and  $y\in\mathcal H^*$ 
set $\varphi_{x,y}(z)=f(z)\langle (zI-T)^{-1}x,y\rangle$ for $z\in\Omega$. 
Set \begin{equation*}\tau=\{\zeta\in\Gamma\ : \ f \text{ has nontangential limit } f(\zeta) \text{ at } \zeta\}.
 \end{equation*}
Then $s(\tau)=s(\Gamma)$. Since  $s(\sigma(T)\cap\Gamma))=0$, we conclude that $\varphi_{x,y}$ has nontangential limit  
$\varphi_{x,y}(\zeta)=f(\zeta)\langle (\zeta I-T)^{-1}x,y\rangle$ for $s$-a.e. $\zeta\in\Gamma$. By \eqref{estkeraa}, 
$\varphi_{x,y}\in L^\infty(\Gamma,s)$. 

Let $x\in\ker A$. By Lemma \ref{lem0},
we have 
\begin{equation*} 0= \int_{\Gamma}z^n \varphi_{x,y}(z)\mathrm{d}z \text{ for every } n\in\mathbb N\cup\{0\}.
\end{equation*}
Since $\Omega$ is a Smirnov domain, we conclude that $\varphi_{x,y}\in H^\infty(\Omega)$ and 
\begin{equation*}\|\varphi_{x,y}\|_{H^\infty(\Omega)}= \|\varphi_{x,y}\|_{L^\infty(\Gamma,s)}.
\end{equation*}
Consequently, if $z\in\Omega$, then 
\begin{equation*}|f(z)||\langle (zI-T)^{-1}x,y\rangle|\leq\mathop\mathrm{{ess\, sup}}_{\zeta\in\Gamma}|f(\zeta)|\|(\zeta I-T)^{-1} x\|\|y\| 
\end{equation*}
for every $y\in\mathcal H^*$. Thus, 
\begin{equation*}\sup_{z\in\Omega}|f(z)|\|(zI-T)^{-1}x\|\leq
 \mathop\mathrm{{ess\, sup}}_{z\in\Gamma}|f(z)|\|(zI-T)^{-1} x\|.
\end{equation*}

Let $x\in\mathcal H$ be such that $\sup_{z\in\Omega}|f(z)|\|(zI-T)^{-1}x\|<\infty$. 
Then $\varphi_{x,y}\in H^\infty(\Omega)$ for every $y\in\mathcal H^*$. Consequently, 
\begin{equation*} \int_\Gamma \varphi_{x,y}(z) \mathrm{d}z=0.
\end{equation*}
Thus, $\langle Ax,y\rangle=0$  for every $y\in\mathcal H^*$. Therefore, $Ax=0$.
\end{proof}

\begin{lemma}\label{lemgamma0} Let $\lambda_0\in\mathbb C$, let $\Gamma,\Gamma_0\colon[0,1]\to\mathbb C$ be two rectifiable Jordan 
curves, and let $\Omega$ and $\Omega_0$ be bounded components of $\mathbb C\setminus \Gamma([0,1])$ and  $\mathbb C\setminus \Gamma_0([0,1])$, 
respectively. Suppose that 
$\Gamma$ and $\Gamma_0$ are positively oriented  (with respect to $\Omega$ and $\Omega_0$),
\begin{equation*} \lambda_0=\Gamma(0)=\Gamma_0(0)=\Gamma(1)=\Gamma_0(1) 
\text{ and }  
(\operatorname{clos}\Omega)\setminus\{\lambda_0\}\subset\Omega_0.
\end{equation*}
Furthermore,  suppose that there exist two sequences 
 $\{t_{0n}\}_{n=1}^\infty$ and  $\{t_{1n}\}_{n=1}^\infty$ such that
\begin{align*}
0<\ldots<t_{0,n+1}<t_{0n}<\ldots<t_{01}<t_{11}<\ldots<t_{n1}<t_{1,n+1}<\ldots<1, \\ t_{0n}\to 0, \ \ t_{1n}\to 1,
\end{align*}
and there exist rectifiable paths $\gamma_{kn}\colon[0,1]\to(\operatorname{clos}\Omega_0)\setminus\Omega$, $n\geq 1$, $k=0,1$, such that
\begin{align*}
\gamma_{0n}(0)=\Gamma(t_{0n}), \ \ \gamma_{0n}(1)=\Gamma_0(t_{0n}), 
\ \ \gamma_{1n}(0)=\Gamma_0(t_{1n}), \ \ \gamma_{1n}(1)=\Gamma(t_{1n}), \\ |\gamma_{kn}|\to_n 0, \ \ k=0,1,
\end{align*}
where by $|\gamma_{kn}|$ the length of the path $\gamma_{kn}$ is denoted, 
 $\gamma_{kn}\bigl((0,1)\bigr)\subset\Omega_0\setminus\operatorname{clos}\Omega$, 
and $\gamma_{kn}(t)=\gamma_{lm}(s)$ for $k,l=0,1$, $m,n\geq 1$ and $t,s\in[0,1]$ if and only if $k=l$, $n=m$ and $t=s$. 
Finally, let $\mathcal U\subset\mathbb C$ be an open set such that $(\operatorname{clos}\Omega_0)\setminus\{\lambda_0\}\subset\mathcal U$, and let $F\colon \mathcal U\to\mathcal L(\mathcal H)$ 
be an analytic (operator-valued) function such that 
 \begin{equation}\label{ffcc} 
 \sup_{z\in(\operatorname{clos}\Omega_0)\setminus(\{\lambda_0\}\cup\Omega)}\|F(z)\|<\infty.
\end{equation}
Then \begin{equation*} \int_\Gamma F(z)\mathrm{d}z=\int_{\Gamma_0} F(z)\mathrm{d}z.
\end{equation*}
\end{lemma}

\begin{proof} Denote by $C$ the supremum from \eqref{ffcc}. 
Set $\Gamma_n=\Gamma|_{[t_{0n},t_{1n}]}$ and $\Gamma_{0n}=\Gamma_0|_{[t_{0n},t_{1n}]}$. 
Then \begin{equation*} \int_\Gamma F(z)\mathrm{d}z=\lim_n\int_{\Gamma_n} F(z)\mathrm{d}z, \ \text{and } 
\  \int_{\Gamma_0} F(z)\mathrm{d}z=\lim_n\int_{\Gamma_{0n}} F(z)\mathrm{d}z.
\end{equation*}
Also, \begin{equation*} \lim_n\int_{\gamma_{kn}}F(z)\mathrm{d}z=0, \ \ k=0,1,
\end{equation*}
because $|\int_{\gamma_{kn}}F(z)\mathrm{d}z|\leq C|\gamma_{kn}|$ and $|\gamma_{kn}|\to_n 0$, $k=0,1$. 

For every $n\geq 1$ set \begin{equation*} \Delta_n=-\Gamma_n\cup\gamma_{0n}\cup\Gamma_{0n}\cup\gamma_{1n},
\end{equation*}
where ``$-$" before $\Gamma_n$ means that the orientation of $\Gamma_n$ must be changed. Then $\Delta_n$ is a
rectifiable Jordan curve, and $F$ is analytic in a neighbourhood of the closure of bounded component of $\mathbb C\setminus\Delta_n$. 
Consequently,
\begin{equation*} 0=\int_{\Delta_n} F(z)\mathrm{d}z=-\int_{\Gamma_n} F(z)\mathrm{d}z+ \int_{\gamma_{0n}}F(z)\mathrm{d}z + 
\int_{\Gamma_{0n}} F(z)\mathrm{d}z+\int_{\gamma_{1n}}F(z)\mathrm{d}z.
\end{equation*}
Tending $n$ to $\infty$, we obtain the conclusion of the lemma.
\end{proof}

\begin{lemma}\label{lemgamma1} Let $\lambda_0\in\mathbb C$, let $\Gamma\colon[0,1]\to\mathbb C$ be a rectifiable Jordan 
curve such that $\lambda_0=\Gamma(0)=\Gamma(1)$, and let $\Omega$ be the bounded component of $\mathbb C\setminus \Gamma([0,1])$. 
Suppose that  $\Gamma_0$, $\Omega_0$, $\mathcal U$ 
 satisfy the assumption of Lemma \ref{lemgamma0}. Let $h\in H^\infty(\mathcal U)$ and $T\in\mathcal L(\mathcal H)$ be such that 
$\sigma(T)\cap\mathcal U=\emptyset$ and 
 \begin{equation}\label{hhcc} \sup_{z\in(\operatorname{clos}\Omega_0)\setminus(\{\lambda_0\}\cup\Omega)}
|h(z)|\|(zI-T)^{-1} \|<\infty.
\end{equation}
For $g\in H^\infty(\mathcal U)$ set
\begin{equation*} A_g=\frac{1}{2\pi\mathrm{i}}\int_\Gamma g(z)h(z)(zI-T)^{-1}\mathrm{d}z.
\end{equation*}
Then $A_g\in\{T\}''$. If $g_1$, $g_2\in H^\infty(\mathcal U)$, then 
\begin{equation*} A_{g_1}A_{g_2}=A_{g_2}A_{g_1}=\frac{1}{2\pi\mathrm{i}}\int_\Gamma g_1(z)g_2(z)h^2(z)(zI-T)^{-1}\mathrm{d}z.
\end{equation*}
\end{lemma} 

\begin{proof} Let $g\in H^\infty(\mathcal U)$. The inclusion $A_g\in\{T\}''$ follows from Lemma \ref{lem0}.  
Let $g_1$, $g_2\in H^\infty(\mathcal U)$. 
By Lemma \ref{lemgamma0} applied to  $F(z)=g_1(z)h(z)(zI-T)^{-1}$ we have 
\begin{equation*} A_{g_1}=\frac{1}{2\pi\mathrm{i}}\int_\Gamma g_1(z)h(z)(zI-T)^{-1}\mathrm{d}z=
\frac{1}{2\pi\mathrm{i}}\int_{\Gamma_0} g_1(z)h(z)(zI-T)^{-1}\mathrm{d}z.
\end{equation*}
Therefore,
\begin{align*}& A_{g_1}A_{g_2}=\frac{1}{(2\pi\mathrm{i})^2}\int_{\Gamma_0} g_1(z)h(z)(zI-T)^{-1}\mathrm{d}z
\int_\Gamma g_2(w)h(w)(wI-T)^{-1}\mathrm{d}w\\&=
\frac{1}{(2\pi\mathrm{i})^2}\int_{\Gamma_0} g_1(z)h(z)\int_\Gamma g_2(w)h(w)
\frac{1}{z-w}((wI-T)^{-1}-(zI-T)^{-1})\mathrm{d}w\mathrm{d}z\\&=
\frac{1}{(2\pi\mathrm{i})^2}\int_{\Gamma_0} g_1(z)h(z)\int_\Gamma g_2(w)h(w)
\frac{1}{z-w}(wI-T)^{-1}\mathrm{d}w\mathrm{d}z,
\end{align*}
because $\int_\Gamma \frac{g_2(w)h(w)}{z-w}\mathrm{d}w=0$ for every  $z\in\Gamma_0\setminus\{\lambda_0\}$. 
Furthermore, since $g_1h\in H^\infty(\Omega_0)$, we have 
$g_1(w)h(w)=\frac{1}{2\pi\mathrm{i}}\int_{\Gamma_0} g_1(z)h(z)\frac{1}{z-w}\mathrm{d}z$ for every $w\in\Omega_0$.
Thus, 
\begin{align*} A_{g_1}A_{g_2}&=\frac{1}{2\pi\mathrm{i}}\int_\Gamma g_2(w)h(w)
\frac{1}{2\pi\mathrm{i}} \int_{\Gamma_0} g_1(z)h(z)\frac{1}{z-w}\mathrm{d}z
(wI-T)^{-1}\mathrm{d}w\\& = \frac{1}{2\pi\mathrm{i}}\int_\Gamma g_2(w)h(w)g_1(w)h(w)(wI-T)^{-1}\mathrm{d}w.
\end{align*}
\end{proof}

\begin{lemma}\label{lemomega12} Let $\Gamma_k$, $k=1,2$,  be two rectifiable Jordan 
curves. Denote by $\Omega_k$ the bounded components of $\mathbb C\setminus \Gamma_k$ 
and by $s_k$ the arc length measure on $\Gamma_k$, $k=1,2$. 
Let $h_k\in E^1(\Omega_k)$,  $k=1,2$, and let $T\in\mathcal L(\mathcal H)$. 
Suppose that 
\begin{align*} \Omega_1\cap\Omega_2=\emptyset, \ \ s_k(\Gamma_1\cap\Gamma_2)=0, 
\ \ s_k(\sigma(T)\cap\Gamma_k)=0, \\  \text{ and } 
\mathop\mathrm{{ess\, sup}}_{z\in\Gamma_k}|h_k(z)|\|(zI-T)^{-1} \|<\infty, \ \ k=1,2,
\end{align*}
(where $\mathop\mathrm{{ess\, sup}}$ is taken with respect to $s_k$, $k=1,2$).
Set
\begin{equation*} A_k=\frac{1}{2\pi\mathrm{i}}\int_{\Gamma_k}h_k(z)(zI-T)^{-1}\mathrm{d}z,\ \ k=1,2.
\end{equation*}
Then $A_k\in\{T\}''$, $k=1,2$, and $A_1A_2=A_2A_1=\mathbb O$.
\end{lemma}

\begin{proof} The inclusions $A_k\in\{T\}''$, $k=1,2$, follow from Lemma \ref{lem0}. We prove the equality
$A_1A_2=\mathbb O$, the equality $A_2A_1=\mathbb O$ is proved similarly. We have  
\begin{align*} (2\pi\mathrm{i})^2 A_1A_2&=\int_{\Gamma_1} h_1(z)(zI-T)^{-1}\mathrm{d}z
\int_{\Gamma_2}h_2(w)(wI-T)^{-1}\mathrm{d}w\\&=
\int_{\Gamma_1} \int_{\Gamma_2}h_1(z) h_2(w)
\frac{1}{z-w}((wI-T)^{-1}-(zI-T)^{-1})\mathrm{d}w\mathrm{d}z\\&=
\int_{\Gamma_2} h_2(w)(wI-T)^{-1}\int_{\Gamma_1} 
\frac{h_1(z)}{z-w}\mathrm{d}z\mathrm{d}w\\&\quad 
-\int_{\Gamma_1} h_1(z)(zI-T)^{-1}\int_{\Gamma_2} 
\frac{h_2(w)}{z-w}\mathrm{d}w\mathrm{d}z.
\end{align*}
If $w\not \in \operatorname {clos}\Omega_1$, then $\int_{\Gamma_1} \frac{h_1(z)}{z-w}\mathrm{d}z=0$. Since 
$s_2(\{w\in\Gamma_2\ : w\in \operatorname {clos}\Omega_1\})=0$, 
we conclude that 
\begin{equation*}\int_{\Gamma_2} h_2(w)(wI-T)^{-1}\int_{\Gamma_1} 
\frac{h_1(z)}{z-w}\mathrm{d}z\mathrm{d}w=0.
\end{equation*}
The equality \begin{equation*}\int_{\Gamma_1} h_1(z)(zI-T)^{-1}\int_{\Gamma_2} 
\frac{h_2(w)}{z-w}\mathrm{d}w\mathrm{d}z=0
\end{equation*}
follows from the similar reasoning.
\end{proof}

\section{Result}

In this section the following theorem is proved.

\begin{theorem}\label{thmbeta} Let $R>1$,  $\beta_1,\beta_2>0$, 
 $\zeta_1,\zeta_2\in\mathbb T$, and let $T\in\mathcal L(\mathcal H)$. Set   
 \begin{equation*}
\Omega_{0k}=\{r\zeta_k\mathrm{e}^{\mathrm{i}t}\ :\ 0<r<R,\ |t|<\beta_k\}, \ \ k=1,2.
\end{equation*}
Suppose that $\Omega_{01}\cap\Omega_{02}=\emptyset$, $\sigma(T)=\{0\}$, and there exist $\beta>\max(\beta_1,\beta_2)$, $C_0$, $c_0>0$ 
such that \begin{equation*}
\|(r\zeta_k\mathrm{e}^{\pm\mathrm{i}\beta_k}I-T)^{-1}\|\leq C_0\exp({c_0}/{r^{\frac{\pi}{2\beta}}}) \text{ for } 0<r<R, \ \ k=1,2,
\end{equation*}
and for every $c>0$ 
\begin{equation}\label{estinfty}
\sup_{z\in\Omega_{0k}}\exp(-{c}/{|z|^{\frac{\pi}{2\beta}}})\|(zI-T)^{-1}\|=\infty, \ \ k=1,2.
\end{equation}
Then there exist  $A_1$, $A_2\in\{T\}''$ such that $A_1A_2=A_2A_1=\mathbb O$ and $A_k^n\neq\mathbb O$ for every $n\in\mathbb N$,  $k=1,2$.
\end{theorem}

\begin{proof} Take $0<\alpha<\min(\beta_1,\beta_2,\pi/2)$ and $C_1>C_0$. 
Apply Lemma \ref{lemgg} to the segments $\zeta_k\mathrm{e}^{\pm\mathrm{i}\beta_k}\cdot(0,1]=\{r\zeta_k\mathrm{e}^{\pm\mathrm{i}\beta_k}\ :\ r\in(0,1]\}$ 
and the function $\varphi(r)=\exp({c_0}/{r^{\frac{\pi}{2\beta}}})$, $r\in(0,1]$. 
Denote obtained sets by $\mathcal G_{k\pm}$, $k=1,2$.  
(More precisely, Lemma \ref{lemgg} is applied to the operator  $\overline{\zeta_k}\mathrm{e}^{\mp\mathrm{i}\beta_k}T$ and the segment $(0,1]$, 
and obtained set  $\mathcal G$ is replaced by $\{\zeta_k\mathrm{e}^{\pm\mathrm{i}\beta_k}z\ : \ z\in\mathcal G\}$.)
Note that the segments  $\zeta_k\mathrm{e}^{\pm\mathrm{i}\beta_k}\cdot(0,1]$
may coincide for different  $k$  and appropriate signs 
``$+$"  and  ``$-$". Consequently, $\mathcal G_{k\pm}$ may coincide. For  example, it is possible that 
$\zeta_1\mathrm{e}^{-\mathrm{i}\beta_1 }=\zeta_2\mathrm{e}^{\mathrm{i}\beta_2}$. Then $\mathcal G_{1-}=\mathcal G_{2+}$. Also,  it is possible that 
$\zeta_1\mathrm{e}^{-\mathrm{i}\beta_1} =\zeta_2\mathrm{e}^{\mathrm{i}\beta_2}$ and 
$\zeta_1\mathrm{e}^{\mathrm{i}\beta_1} =\zeta_2\mathrm{e}^{-\mathrm{i}\beta_2}$. Then 
 $\mathcal G_{1-}=\mathcal G_{2+}$ and $\mathcal G_{1+}=\mathcal G_{2-}$.

  Depending of $\beta_k$ and $\zeta_k$, $k=1,2$, two, three, or four different sets
$\mathcal G_{k\pm}$ are obtained.  
Set \begin{equation*}\mathcal G_k=\mathcal G_{k+}\cup\mathcal G_{k-},  \ \ 
\Omega_k=(\Omega_{0k}\cap\mathbb D)\setminus\mathcal G_k 
\text{ and } \Gamma_k=\partial\Omega_k, \ \ k=1,2. 
\end{equation*}
Then $\Gamma_k$, $k=1,2$, are chord-arc rectifiable Jordan curves, that is, Lavrentiev curves, due to the construction of the sets $\mathcal G_{k\pm}$ and the choice of $\alpha$. (The definition and references are recalled in Sec. 2, and the details  left to the reader.) 
 Consequently,  $\Omega_k$, $k=1,2$, are Smirnov domains by   {\cite[Theorem 7.11]{pom92}}.

Set 
\begin{align*} c_k&={c_0}/{\cos\bigl(\frac{\pi}{2}\frac{\beta_k}{\beta}\bigr)} \text{ and } \\
h_k(z)&=\exp(-{c_k}/{(\overline{\zeta_k} z)^{\frac{\pi}{2\beta}}})
\text{ for }z=r\zeta_k\mathrm{e}^{\mathrm{i}t}\ :\ r>0,\ |t|<\pi, \ \ k=1,2.
\end{align*}
Then $h_k\in H^\infty(\Omega_{0k})$, $k=1,2$, and 
\begin{equation}\label{omegatt}\sup_{z\in(\operatorname{clos}\Omega_{0k})\setminus(\{0\}\cup\Omega_k)}|h_k(z)|\|(zI-T)^{-1} \|<\infty, \ \ k=1,2.
\end{equation} 
Indeed, let $z= r\zeta_k\mathrm{e}^{\mathrm{i}t}$, $0<r<R$, $|t|<\beta_k$. 
Then \begin{equation*}
|h_k(z)|=\exp\bigl(-{c_k}\cos\bigl(\frac{\pi}{2}\frac{t}{\beta}\bigr)/r^{\frac{\pi}{2\beta}}\bigr)
\leq \exp\bigl(-{c_0}/r^{\frac{\pi}{2\beta}}\bigr), 
\ \ k=1,2.
\end{equation*}
Furthermore, if $z=r\zeta_k\mathrm{e}^{\mathrm{i}t}\in\mathcal G_k$, then  
\begin{equation*}
|h_k(z)|\|(zI-T)^{-1} \|\leq\exp(-{c_0}/{r^{\frac{\pi}{2\beta}}})C_1
\exp({c_0}/{r^{\frac{\pi}{2\beta}}})=C_1, \ \ k=1,2,
\end{equation*}
by Lemma \ref{lemgg}. 
Since $\sigma(T)=\{0\}$, we have \begin{equation*}
\sup_{z= r\zeta_k\mathrm{e}^{\mathrm{i}t}, \, 1\leq r\leq R,\,  |t|\leq \beta_k}|h_k(z)|\|(zI-T)^{-1} \|<\infty, \ \ k=1,2.
\end{equation*} 

Thus, $T$, $h_k$, $\Omega_k$, $k=1,2$, satisfy the assumptions of Lemma \ref{lemomega12}. 
Furthermore, $T$, $h_1$, $\Omega_1$, $\Omega_{01}$ 
satisfy the assumptions of Lemma \ref{lemgamma1} with $\lambda_0=0$, and $T$, $h_2$, $\Omega_2$, $\Omega_{02}$ 
satisfy the assumptions of Lemma \ref{lemgamma1} with $\lambda_0=0$, too.  The segments 
  $\{\zeta_k\mathrm{e}^{\pm\mathrm{i}\beta_k}\cdot\bigl(a_{n+1}\pm\mathrm{i}\cdot[0,\delta_n]\bigr)\}_{n=1}^\infty$ , where  $\{a_n\}_{n=1}^\infty$ and  $\{\delta_n\}_{n=1}^\infty$ are from the construction of the correspondent set  
 $\mathcal G$, can serve as  $\{\gamma_{ln}\}_{n=1}^\infty$, $l=0,1$, with appropriate choice of indices and sings. 

Set 
\begin{equation*} A_k=\frac{1}{2\pi\mathrm{i}}\int_{\Gamma_k}h_k(z)(zI-T)^{-1}\mathrm{d}z,\ \ k=1,2.
\end{equation*}
By Lemma \ref{lemomega12}, $A_k\in\{T\}''$, $k=1,2$, and $A_1A_2=A_2A_1=\mathbb O$. By  Lemma \ref{lemgamma1},  
\begin{equation*} A_k^n=\frac{1}{2\pi\mathrm{i}}\int_{\Gamma_k} h_k^n(z)(zI-T)^{-1}\mathrm{d}z,\ \ k=1,2, \text{ for every }n\in\mathbb N.
\end{equation*}
By Lemma \ref{lemkeraa},
\begin{equation*} \ker A_k^n=\{x\in\mathcal H\ : \ \sup_{z\in\Omega_k}|h_k^n(z)|\|(zI-T)^{-1}x\|<\infty\}.
\end{equation*}
Let $n\in\mathbb N$, and let $k=1$ or $k=2$. Assume that $A_k^n=\mathbb O$. Then  
\begin{equation*} \sup_{z\in\Omega_k}|h_k^n(z)|\|(zI-T)^{-1}x\|<\infty \text{ for every } x\in\mathcal H. 
\end{equation*}
By the Banach--Steinhause Theorem (applied to the family of operators $\{h_k^n(z)(zI-T)^{-1}\}_{z\in\Omega_k}$), 
 \begin{equation*} \sup_{z\in\Omega_k}|h_k^n(z)|\|(zI-T)^{-1}\|<\infty. 
\end{equation*}
From the latest estimate and \eqref{omegatt} we conclude that 
\begin{equation*}
C=\sup_{z\in\Omega_{0k}}|h_k^n(z)|\|(zI-T)^{-1} \|<\infty.
\end{equation*}
This contradicts to \eqref{estinfty}. Indeed, let $z= r\zeta_k\mathrm{e}^{\mathrm{i}t}$, $0<r<R$, $|t|<\beta_k$. 
Then 
\begin{equation*}\|(zI-T)^{-1} \|\leq\frac{C}{|h_k^n(z)|}=
C\exp\bigl(n{c_k}\cos\bigl(\frac{\pi}{2}\frac{t}{\beta}\bigr)/r^{\frac{\pi}{2\beta}}\bigr)
\leq C\exp\bigl(n{c_k}/r^{\frac{\pi}{2\beta}}\bigr).
\end{equation*}
Consequently, \eqref{estinfty} is not fulfilled for $c=nc_k$, a contradiction. Thus, $A_k^n\neq\mathbb O$ for every $n\in\mathbb N$,  $k=1,2$.
\end{proof}

\begin{corollary} Let $T\in\mathcal L(\mathcal H)$ satisfy the assumptions of Theorem \ref{thmbeta}. 
Then  $\operatorname{Hlat}T$ is not totally ordered by inclusion. In particular, every operator from $\{T\}'$ is not unicellular. 
\end{corollary}

\begin{proof} The conclusion of  the corollary follows from Theorem \ref{thmbeta} and Lemma \ref{leminv}.
\end{proof}

\subsection*{Comparison with results from \cite{sol}}

One of results from \cite{sol} is as follows. 

\begin{theoremcite}\label{thmsol}
 Let $N\in\mathbb N$,  $0<\beta\leq\pi$,  
 $\zeta_0\in\mathbb T$, and let $T\in\mathcal L(\mathcal H)$. 
Suppose that $\sigma(T)=\{0\}$, and there exist $K$,  $C_0$, $c_0>0$ 
such that \begin{equation}\label{solbeta}
\|(r\zeta_0\mathrm{e}^{\mathrm{i}t}I-T)^{-1}\|\leq C_0\exp({c_0}/{r^{\frac{\pi}{2\beta}}}) \text{ for } 0<r<1, \ |t|<\beta,
\end{equation}
and 
\begin{equation}\label{solnn}
\|(r\zeta_0\mathrm{e}^{\mathrm{i}t}I-T)^{-1}\|\leq{K}/{r^N} \text{ for } 0<r<1, \ \beta\leq|t|\leq\pi.
\end{equation}
Then $T$ has a nontrivial hyperinvariant subspace.
\end{theoremcite}

Suppose that \eqref{solnn} if fulfilled for some $T\in\mathcal L(\mathcal H)$ with $\sigma(T)=\{0\}$, but  
there is no  $C_0$, $c_0>0$ such that \eqref{solbeta} is  fulfilled. Then \eqref{estinfty} is fulfilled 
for $\Omega=\{\zeta_0\mathrm{e}^{\mathrm{i}t}\ :\ 0<r<1,\ |t|<\beta\}$. In addition, suppose that 
there exists $t_0$ such that $|t_0|<\beta$
and 
 \begin{equation*}
\|(r\zeta_0\mathrm{e}^{\mathrm{i}t_0}I-T)^{-1}\|\leq C_0\exp({c_0}/{r^{\frac{\pi}{2\beta}}}) \text{ for } 0<r<1. 
\end{equation*}
Set 
 \begin{align*}
\Omega_{01}&=\{\zeta_0 r\mathrm{e}^{\mathrm{i}t}\ :\ 0<r<R,\ -\beta<t<t_0\}  \\ \text{ and }
 \Omega_{02}&=\{\zeta_0 r\mathrm{e}^{\mathrm{i}t}\ :\ 0<r<R,\ t_0<t<\beta\}
\end{align*}
for some $R>1$. 
Since \eqref{estinfty} is fulfilled for $\Omega$, we have \eqref{estinfty} is fulfilled for at least one of two domains $\Omega_{0k}$,  $k=1,2$.  
If  \eqref{estinfty} is fulfilled for \emph{both} $\Omega_{01}$ and  $\Omega_{02}$, then  $T$ has a nontrivial hyperinvariant subspace by 
Theorem \ref{thmbeta}.  But  if  \eqref{estinfty} is fulfilled for \emph{only one} of $\Omega_{0k}$,  $k=1,2$, 
then  Theorem  \ref{thmsol} and  Theorem \ref{thmbeta} cannot be applied.

 If $T$ is a power bounded operator, that is,  $K_0=\sup_{n\in\mathbb N}\|T^n\|<\infty$, then 
$\|(zI-T)^{-1}\|\leq{K_0}/(|z|-1)$ for $|z|>1$. Consequently, 
\eqref{solnn} is fulfilled for $T-I$ with $N=2$, $\beta=\pi/2$ and $\zeta_0=-1$.
Indeed, set $w=z-1$. If  $\operatorname{Re}w>0$ and $|w|<1$, then $|w+1|-1\geq|w|^2/3$. Consequently, 
 \begin{equation*}\|(wI-(T-I))^{-1}\|\leq{K_0}/(|w+1|-1)\leq 3K_0/|w|^2.\end{equation*} 
If $\sigma(T)=\{1\}$, then $\sigma(T-I)=\{0\}$.  Theorem  \ref{thmsol} and  Theorem \ref{thmbeta} give sufficient conditions for $T$ have nontrivial hyperinvariant subspaces.

\end{document}